\documentclass{amsart}
\usepackage{mathabx}
\usepackage{mathrsfs}
\usepackage{enumerate}
\usepackage{color}

\newtheorem{thm}{Theorem}
\newtheorem{lem}[thm]{Lemma}
\newtheorem{cor}[thm]{Corollary}
\newtheorem{defi}[thm]{Definition}
\newtheorem{rem}[thm]{Remark}

\newtheorem{quest}[thm]{Question}

\begin{document}

\title{Catalysis in the trace class and weak trace class ideals}
\author{Guillaume Aubrun}
\address{Institut Camille Jordan, Universit\'e Claude Bernard Lyon 1, 43 boulevard du 11 novembre 1918,
69622 Villeurbanne cedex, France}
\email{aubrun@math.univ-lyon1.fr}
\author{Fedor Sukochev}
\address{School of Mathematics and Statistics, University of NSW, Sydney,  2052, Australia}
\email{f.sukochev@unsw.edu.au}
\author{Dmitriy Zanin}
\address{School of Mathematics and Statistics, University of NSW, Sydney,  2052, Australia}
\email{d.zanin@unsw.edu.au}

\thanks{The research of GA was supported by the ANR projects OSQPI (ANR-11-BS01-0008) and StoQ (ANR-14-CE25-0003). The research of FS and DZ has been supported by the ARC projects DP140100906 and DP 120103263.}


\begin{abstract}
Given operators $A,B$ in some ideal $\mathcal{I}$ in the algebra $\mathcal{L}(H)$ of all bounded operators on a 
separable Hilbert space  $H$, can we give conditions guaranteeing 
the existence of a trace-class operator $C$
such that $B \otimes C$ is submajorized (in the sense of Hardy--Littlewood) by $A \otimes C$ ? In the case when 
$\mathcal{I} = \mathcal{L}_1$, a necessary and almost
sufficient condition is that the inequalities ${\rm Tr} (B^p) \leq {\rm Tr} (A^p)$ hold for every $p \in [1,\infty]$. We show
that the analogous statement fails for $\mathcal{I} = \mathcal{L}_{1,\infty}$ by connecting it with the study of Dixmier traces.
\end{abstract}

\maketitle

\section{Introduction}

Let $H$ be an infinite-dimensional separable Hilbert space, $\mathcal{L}(H)$ be the algebra of all bounded operators on $H$ and
$\mathcal{C}_0=\mathcal{C}_0(\mathcal{H})$ the set of compact operators. 

Given $A \in \mathcal{C}_0$, we denote by 
$\mu(A):=\{\mu(k,A)\}_{k \geq 0}$ the sequence of singular values of the operator $A$ 
(that is, eigenvalues of the operator $|A|$) arranged in the decreasing order and taken with multiplicities (if any).
We say that $B\in \mathcal{C}_0$ is submajorized by $A\in \mathcal{C}_0$ in the sense of Hardy--Littlewood 
(written $B\prec\prec A$) if for every integer $n$
$$\sum_{k=0}^{n}\mu(k,B)\leq\sum_{k=0}^{n}\mu(k,A).$$

If $A,B \in \mathcal{C}_0$ are such that $B \prec \prec A$, then $B \otimes C \prec \prec A \otimes C$ for every
$C \in \mathcal{C}_0.$\footnote{Suppose first that $C\geq0$ has finite rank.  That is, $C=\sum_{k=0}^{n-1}\mu(k,C)p_k,$ where $p_k,$ $0\leq k<n,$ are pairwise orthogonal rank one projections. Set $A_k=A\otimes\mu(k,C)p_k$ and $B_k=B\otimes\mu(k,C)p_k.$ It is immediate that $B_k\prec\prec A_k$ for $0\leq k<n.$ It follows from Lemma 2.3 in \cite{CDS} that $\sum_{k=0}^{n-1}B_k\prec\prec\sum_{k=0}^{n-1}A_k$ or, equivalently, $B\otimes C\prec\prec A\otimes C.$ For an arbitrary $C,$ the assertion follows by approximation.  }The converse does not hold, even in the finite-dimensional
setting: if $A,B,C$ are such that $\mu(A)=(0.5,0.25,0.25,0,\cdots)$, $\mu(B) = (0.4,0.4,0.1,0.1,0,\cdots)$ and 
$\mu(C)=(0.6,0.4,0,\cdots)$, 
one checks easily that $B \otimes C \prec \prec A \otimes C$ while $B$ is not submajorized by $A$. This 
example appears in \cite{JP} and is related to the phenomenon of \emph{catalysis} in quantum information theory (the
operator $C$ being called a catalyst). This corresponds to the situation where the transformation of some quantum state 
(in that case, $B$) into another quantum state (in that case, $A$) is only possible in the presence of an extra quantum state 
(in that case, $C$) although the latter is not consumed in the process.
It is argued in \cite{JP} that this phenomenon can be used to improve the efficiency of entanglement
concentration procedures.

In the following we restrict ourselves to $A,B$ being positive elements in $\bigcap_{p>1} \mathcal{L}_p$ ($\mathcal{L}_p$ denoting 
the Schatten--von Neumann ideal) and compare the following statements.
\begin{enumerate}[{\rm (i)}]
\item\label{cim2} There exists a nonzero $C \in \mathcal{L}_1$ such that $B \otimes C \prec \prec A \otimes C$.
\item\label{cim3} For every $p>1,$ we have ${\rm Tr} (B^p) \leq {\rm Tr} (A^p)$.
\end{enumerate}


One checks  that \eqref{cim2} implies \eqref{cim3}. This follows from the monotonicity of $A \mapsto {\rm Tr} (A^p)$ with respect
to submajorization and from the formula
$${\rm Tr} (S \otimes T) = {\rm Tr} (S) \cdot {\rm Tr} (T),\quad S,T \in \mathcal{L}_1.$$
There is some hope to reverse the implication $\eqref{cim2}\Rightarrow\eqref{cim3}$ if we 
 allow  closure of the set
$$\{B:\exists C \in \mathcal{L}_1\mbox{ such that }B \otimes C \prec \prec A \otimes C\}$$
with respect to some topology (for the finite-dimensional case, see \cite{Au-Nec,Klimesh,Turgut}).

To explain why some closure is needed, we give an example of a pair $A,B$ of positive operators satisfying (ii) but
not (i). Consider positive operators with $\mu(A) \neq \mu(B)$ and 
such that ${\rm Tr} (B^p) \leq {\rm Tr} (A^p)$ for $p \in (1,\infty)$, 
while ${\rm Tr} (B^{p_0}) = {\rm Tr} (A^{p_0})$ for some $p_0 \in (1,\infty)$ (such an example exists among
finite rank operators). 
Note that the norm in $\mathcal{L}_{p_0}$ is strictly monotone with respect to submajorization 
(see Proposition 2.1 in \cite{CDSS}). That is, if $K\in\mathcal{L}_{p_0}(H)$ and if 
$L\prec\prec K,$ then either $\mu(L)=\mu(K)$ or $\|L\|_{p_0}<\|K\|_{p_0}.$ Suppose that (i) holds, i.e. that
$B \otimes C \prec \prec A \otimes C$ for some nonzero $C \in \mathcal{L}_1$ (that is, no closure is taken). 
We then have
${\rm Tr} ((B \otimes C)^{p_0}) = {\rm Tr} ((A \otimes C)^{p_0})$ and, by strict monotonicity, $\mu(k,B \otimes C) = \mu(k,A \otimes C)$ for all $k\geq0.$ Now, taking into account that the sequences  $\mu(B \otimes C)$  and 
$\mu(A \otimes C)$ coincide with decreasing rearrangements of sequences  $\mu(B )\otimes \mu(C)$   and    $\mu(A )\otimes \mu(C)$  respectively, we infer that 
$\mu(A)=\mu(B).$ 

As we shall see, the choice of the topology plays a crucial role. 
Prior to stating the precise question, we recall a few definitions and relevant facts.

There is a remarkable correspondence between sequence spaces and two-sided ideals in $\mathcal{L}(H)$ due to J.W.~Calkin, 
\cite{Calkin}.  Recall,
that a linear subspace $\mathcal J$ in $\mathcal{L}(H)$ is a two-sided ideal if $X \in \mathcal {J}$ and $Y \in \mathcal{L}(H)$ imply $YX, XY \in \mathcal J.$ Every non-trivial ideal necessarily consists of compact operators. A Calkin space $J$ is a subspace of $c_0$ (the space of all vanishing sequences) such
that $x \in J$ and $\mu(y)\leq\mu(x)$ imply $y \in J$, where $\mu(x)$ is the decreasing rearrangement of the sequence $|x|$. The Calkin correspondence may be explained as follows. If $J$ is a
Calkin space then associate to it the subset $\mathcal J$ in $\mathcal{L}(H)$
$$\mathcal{J}:=\{ X \in \mathcal C_0 : \mu(X) \in J \} .$$
Conversely, if $\mathcal J$ is a two-sided ideal, then associate to it the sequence space
$$J:=  \{ x \in c_0 : \mu(x)=\mu(X)\ \mbox{for some}\ X \in \mathcal J \} .$$
For the proof of the following theorem we refer to Calkin's original paper, \cite{Calkin}, and to B.~Simon's book, \cite[Theorem 2.5]{Simon}.

\begin{thm}[Calkin correspondence]\label{Calkin correspondence}
The correspondence $J \leftrightarrow \mathcal J$ is a bijection between Calkin spaces and
two-sided ideals in $\mathcal{L}(H).$
\end{thm}
In the recent papers \cite{KS}, \cite{S}  this correspondence has been specialised to quasi-normed symmetrically-normed
ideals and quasi-normed symmetric sequence spaces \cite{LT1}. We use the notation $\|\cdot\|_\infty$ to denote the
uniform norm on $\mathcal{L}(H)$.

\begin{defi}\label{sn ideal}
\begin{enumerate}[{\rm (i)}]
\item An ideal $\mathcal{E}$ in $\mathcal{L}(H)$ is said to be symmetrically (quasi)-normed
if it is equipped with a Banach (quasi)-norm $\|\cdot\|_{\mathcal{E}}$ such that
$$\|XY\|_{\mathcal{E}},\|YX\|_{\mathcal{E}}\leq\|X\|_{\mathcal{E}}\|Y\|_{\infty},\quad X\in\mathcal{E},Y\in \mathcal{L}(H).$$
\item A Calkin space $E$ is a symmetric sequence space if it is equipped with a Banach (quasi)-norm $\|\cdot\|_{{E}}$ such that
$\|y\|_E\leq \|x\|_E$ for every $x\in E$ and $y\in c_0$ such that
$\mu(y)\leq\mu(x)$.
\end{enumerate}
\end{defi}

For convenience of the reader, we recall that a map
$\|\cdot\|$ from a linear space $X$ to $\mathbb{R}$ is a quasi-norm, if for all
$x,y\in X$ and scalars $\alpha$ the following properties hold:
\begin{enumerate}[{\rm (i)}]
\item $\|x\|\geqslant 0$, and $\|x\|=0 \Leftrightarrow x=0$;
\item $\|\alpha x\|=|\alpha|\|x\|$;
\item $\|x+y\|\leqslant C(\|x\|+\|y\|)$ for some $C\geqslant 1$.
\end{enumerate}
The couple $(X,\|\cdot\|)$ is a quasi-normed space and the least
constant $C$  satisfying the  inequality (iii) above is called the
modulus of concavity of the quasi-norm $\|\cdot\|$ and denoted by
$C_X$. A complete quasi-normed space is called quasi-Banach.

It easily follows from Definition \ref{sn ideal} that if $(\mathcal{E}, \|\cdot\|_{\mathcal{E}})$ is a quasi-Banach ideal, $X\in\mathcal{E}$ and  $Y\in \mathcal{L}(H)$ are such
that $\mu(Y)\leq\mu(X)$, then $Y\in\mathcal{E}$ and
$\|Y\|_{\mathcal{E}}\leq\|X\|_{\mathcal{E}}$. In particular, it is
easy to see that if $E$ is Calkin space corresponding to
$\mathcal{E}$, then setting $\|x\|_E:=\|X\|_{\mathcal{E}}$ (where
$X\in \mathcal{E}$ is such that $\mu(x)=\mu(X)$) we obtain that
$(E, \|\cdot\|_E)$ is a quasi-Banach symmetric sequence space. The converse
implication is much harder and follows from main results of \cite{KS, S}. 

With these preliminaries out of the way, we are now in a position to formulate the main question.

\begin{quest}\label{meaningful question} Let $\mathcal{I}$ be a (quasi-)Banach ideal such that 
$\mathcal{I}\subset\bigcap_{p>1}\mathcal{L}_p.$ Let $0\leq A\in\mathcal{I}.$ Consider the sets
$${\rm PM}(A,\mathcal{I})=\Big\{0\leq B\in\mathcal{I}:\ {\rm Tr}(B^p)\leq{\rm Tr}(A^p)\ \forall p>1\Big\}.$$
$${\rm Catal}(A,\mathcal{I})=\Big\{0\leq B\in\mathcal{I}:\ \exists0\leq C\in\mathcal{L}_1:\ C \neq 0, \ B\otimes C\prec\prec A\otimes C\Big\}.$$
Let also $\overline{\rm Catal}(A,\mathcal{I})$ denote the closure of ${\rm Catal}(A,\mathcal{I})$ with respect to the quasi-norm of $\mathcal{I}.$ Is it true that ${\rm PM}(A,\mathcal{I}) = \overline{\rm Catal}(A,\mathcal{I})?$
\end{quest}

Note that ${\rm PM}(A,\mathcal{I})$ is a closed subset in $\mathcal{I}.$ Indeed, let $B_n\in{\rm PM}(A,\mathcal{I})$ and let $B_n\to B$ in $\mathcal{I}$ as $n\to\infty.$ 
Observe that it follows from Definition \ref{sn ideal} that $\mathcal{I}$ is continuously embedded into $\mathcal{L}(H)$ and therefore, it follows from the Closed Graph Theorem that  for every fixed $p>1,$  the identical embedding $\mathcal{I}\subset \mathcal{L}_p$ is continuous, in particular, there exists a constant $c(p,\mathcal{I})$ such that $\|C\|_p\leq c(p,\mathcal{I})\|C\|_{\mathcal{I}},$ $C\in\mathcal{I}.$ Thus,
$$\Big|\|B_n\|_p-\|B\|_p\Big|\leq\|B-B_n\|_p\leq c(p,\mathcal{I})\|B-B_n\|_{\mathcal{I}}\to0.$$ 
Hence,
$${\rm Tr}(B^p)=\lim_{n\to\infty}{\rm Tr}(B_n^p)\leq{\rm Tr}(A^p),\quad p>1.$$
We also have that ${\rm Catal}(A,\mathcal{I})\subset {\rm PM}(A,\mathcal{I}).$ Indeed, if $B\otimes C\prec\prec A\otimes C,$ then
$${\rm Tr}(B^p)=\frac{{\rm Tr}((B\otimes C)^p)}{{\rm Tr}(C^p)}\leq\frac{{\rm Tr}((A\otimes C)^p)}{{\rm Tr}(C^p)}={\rm Tr}(A^p),\quad p>1.$$
Since ${\rm PM}(A,\mathcal{I})$ is closed, it follows that the inclusion $\overline{\rm Catal}(A,\mathcal{I}) \subset {\rm PM}(A,\mathcal{I})$ always holds.

In this paper, we show that the answer to Question \ref{meaningful question} is positive when $\mathcal{I} = \mathcal{L}_1$ and negative when $\mathcal{I} = \mathcal{L}_{1,\infty}.$ Recall that $\mathcal{L}_{1,\infty}$ is the principal ideal generated by the element $A_0={\rm diag}(\{1,\frac12,\frac13,\cdots\}).$ Equivalently, 
$$\mathcal{L}_{1,\infty} = \{ A \in \mathcal{C}_0 : \ \sup_{k \geq 0} (k+1)\mu(k,A) < +\infty \}.$$
It becomes a quasi-Banach space (see e.g. \cite{KS,S}) when equipped with the quasi-norm
$$\|A\|_{1,\infty}=\sup_{k\geq0}(k+1)\mu(k,A),\quad A\in\mathcal{L}_{1,\infty}.$$

Here are our main results. We leave open the question of giving a complete description of the set $\overline{\rm Catal}(A,\mathcal{L}_{1,\infty})$.

\begin{thm}\label{thm-L1}
For every $0 \leq A \in \mathcal{L}_1$, the sets ${\rm PM}(A,\mathcal{L}_1)$ and $\overline{\rm Catal}(A,\mathcal{L}_1)$ coincide.
\end{thm}

\begin{thm}\label{main}
There exists $0 \leq A \in \mathcal{L}_{1,\infty}$ such that the set ${\rm PM}(A,\mathcal{L}_{1,\infty})$ strictly contains the set $\overline{\rm Catal}(A,\mathcal{L}_{1,\infty}).$
\end{thm}

It is actually simple to deduce Theorem \ref{thm-L1} from the finite-dimensional considerations from \cite{Au-Nec}, as we explain in Section \ref{section:L1}. 
This is in sharp contrast with Theorem \ref{main}, whose proof is infinite-dimensional in its nature and uses crucially fine properties of Dixmier traces, which we introduce
in Section \ref{section:dixmier}. The heart of the argument behind Theorem \ref{main} appears in Section \ref{section:proof}, while we relegate some needed computations to Section \ref{section:computations}.

The authors thank the anonymous referees for careful reading of the manuscript and numerous suggestions which have improved the exposition.

\section{The case of $\mathcal{L}_1$} \label{section:L1}

We derive Theorem \ref{thm-L1} from the following result which appears in \cite{Turgut} (see also Lemma 2 in \cite{Au-Nec}).

\begin{lem} \label{lem-aunec}
Let $A,B$ be positive finite rank operators. Assume that for every $1 \leq p \leq +\infty,$ we have the strict inequality $\|B\|_p<\|A\|_p.$ Then there exists a nonzero finite rank operator $C$ such that $B \otimes C \prec \prec A \otimes C.$
\end{lem}

\begin{proof}[Proof of Theorem \ref{thm-L1}] Let us show the non-trivial inclusion, i.e. that every $B \in {\rm PM}(A,\mathcal{L}_1)$ belongs to 
$\overline{\rm Catal}(A,\mathcal{L}_1).$

Let $p_k,$ $k\geq0,$ be a rank one eigenprojection of the operator $A$ which corresponds to the eigenvalue $\mu(k,A).$ Similarly, let $q_k,$ $k\geq0,$ be a rank one eigenprojection of the operator $B$ which corresponds to the eigenvalue $\mu(k,B).$ We have
$$A=\sum_{k=0}^{\infty}\mu(k,A)p_k,\quad B=\sum_{k=0}^{\infty}\mu(k,B)q_k.$$

Without loss of generality, $\mu(0,A)=1.$ It follows that 
$$(1-(1-\varepsilon)^p){\rm Tr}(A^p)\geq(1-(1-\varepsilon)^p)\mu(0,A)^p=1-(1-\varepsilon)^p\geq\varepsilon^p.$$
The latter readily implies
$${\rm Tr}(A^p)-\varepsilon^p\geq(1-\varepsilon)^p{\rm Tr}(A^p),\quad p\geq1,\quad\varepsilon\in(0,1).$$

Now, fix $\varepsilon\in(0,1)$ and select $n$ such that
$$\sum_{k=n}^{\infty}\mu(k,A)<\varepsilon,\quad \sum_{k=n}^{\infty}\mu(k,B)<\varepsilon.$$
Set
$$A_n=\sum_{k=0}^{n-1}\mu(k,A)p_k,\quad B_n=\sum_{k=0}^{n-1}\mu(k,B)q_k.$$
It is clear that
$${\rm Tr}(A_n^p)={\rm Tr}(A^p)-\sum_{k=n}^{\infty}\mu(k,A)^p\geq{\rm Tr}(A^p)-(\sum_{k=n}^{\infty}\mu(k,A))^p$$
$$>{\rm Tr}(A^p)-\varepsilon^p\geq (1-\varepsilon)^p{\rm Tr}(A^p),\quad p\geq1.$$
Therefore,
$$(1-\varepsilon)^p{\rm Tr}(B_n^p)\leq(1-\varepsilon)^p{\rm Tr}(B^p)\leq(1-\varepsilon)^p{\rm Tr}(A^p)<{\rm Tr}(A_n^p),\quad p\geq 1.$$
Since both $A_n$ and $B_n$ are finite rank operators, it follows from Lemma \ref{lem-aunec} 
and the first footnote that there exists a finite rank operator $C_n$ such that
$$(1-\varepsilon)B_n\otimes C_n\prec\prec A_n\otimes C_n\prec\prec A\otimes C_n.$$
In particular, we have that $(1-\varepsilon)B_n\in{\rm Catal}(A,\mathcal{L}_1).$ 
Observing that $\|B-B_n\|_1\leq 1$, we further obtain
$$\|B-(1-\varepsilon)B_n\|_1\leq\varepsilon\|B\|_1+(1-\varepsilon)\|B-B_n\|_1\leq\varepsilon(\|B\|_1+1).$$
Since $\varepsilon$ is arbitrarily small, it follows that $B\in\overline{\rm Catal}(A,\mathcal{L}_1).$
\end{proof}

\section{Dixmier traces} \label{section:dixmier}

The crucial ingredient in the proof is the notion of a Dixmier trace on $\mathcal{L}_{1,\infty}$. 
Let $\ell_{\infty}$ stand for the Banach space of all bounded sequences $x=(x_n)_{n\ge 0}$ 
equipped with the usual norm $\|x\|_\infty:=\sup _{n\ge 0}|x_n|$. A \emph{generalized limit} is any positive linear 
functional on $\ell_{\infty}$ which equals the ordinary limit on 
the subspace  $c$ of all convergent sequences.

\begin{rem} Given a sequence $(x_n)_{n\geq0}\in \ell_\infty$, there is a generalized limit $\omega$ such that $\omega ((x_n))= \limsup_{n\to \infty} x_n$.
\end{rem}
\begin{proof} Fix $x=(x_n)\in \ell_\infty$ and let the sequence $(n_k)_{k\ge 0}$ be such that $\lim_{k\to \infty} x_{n_k}=\limsup_{n\to \infty} x_n$. Consider the set of 
functionals $(\varphi_{n_k})_{k\ge 0}$ on $\ell_{\infty}$  defined by $\varphi_{n_k} (y_n):=y_{n_k}$, $y=(y_n)\in \ell_{\infty}$, $k\ge 0$. The set  $(\varphi_{n_k})_{k\ge 0}$ belongs to the unit ball $B$ of the Banach dual $\ell_{\infty}^*$. The set $B$ is compact in the weak$^*$ topology $\sigma(\ell_{\infty}^*,\ell_{\infty})$ and therefore the set  $(\varphi_{n_k})_{k\ge 0}$ possesses a cluster point  $\omega\in  \ell_{\infty}^*$  in that topology.  The fact that $\omega$ is a generalized limit  on $\ell_{\infty}$ such that $\omega ((x_n))= \limsup_{n\to \infty} x_n$ follows immediately from the definition of the weak$^*$ topology.
\end{proof}

The Dixmier traces are defined as follows.

\begin{thm} \label{dixmier-traces} Let $\omega$ be a generalized limit. The mapping ${\rm Tr}_{\omega}:\mathcal{L}_{1,\infty}^+\to\mathbb{R}^+$ 
defined for $0 \leq A \in\mathcal{L}_{1,\infty}$ by setting
$$ {\rm Tr}_{\omega}(A):= \omega \left( \left\{ \frac1{\log(N+2)}\sum_{k=0}^N\mu(k,A) \right\}_{N=0}^{\infty} \right)$$
is additive and, therefore, extends to a positive unitarily invariant linear functional on $\mathcal{L}_{1,\infty}$ called 
a Dixmier trace.
\end{thm}

Note that the positivity of generalized limits implies that
\begin{equation}\label{continuity}
\left| {\rm Tr}_{\omega}(A) \right| \leq \|A\|_{1,\infty}
\end{equation}
for every Dixmier trace ${\rm Tr}_{\omega}$ and $A \in \mathcal{L}_{1,\infty}$.

Let us comment on how additivity is proved in Theorem \ref{dixmier-traces}. This is usually achieved under the extra assumption that $\omega$ is scale invariant (see Theorem 1.3.1 in \cite{LSZ}), 
i.e. that $\omega \circ \sigma_k = \omega$ for all positive integer $k$, where 
$\sigma_k : \ell_{\infty} \to \ell_{\infty}$ is defined as
$$ \sigma_k(x_1,x_2,\dots,x_n,\dots) = (\underbrace{x_1,\dots,x_1}_{k \textnormal{ times}},
\underbrace{x_2,\dots,x_2}_{k \textnormal{ times}},\dots, \underbrace{x_n,\dots,x_n}_{k \textnormal{ times}},\dots).$$
Under this extra assumption the map ${\rm Tr}_{\omega}$ is actually additive on the larger ideal $\mathcal{M}_{1,\infty}$ 
(we refer to \cite[Example 1.2.9]{LSZ} for the definition of the latter ideal and to \cite[Section 6.8 ]{LSZ} for
historical background). In the form presented here, 
Theorem \ref{dixmier-traces} follows from Theorem 17 in \cite{SS}. For the reader's convenience we reproduce the argument here.

\begin{proof}[Proof of Theorem \ref{dixmier-traces}]
Given $A \in \mathcal{L}_{1,\infty}$, consider the sequence $(x_N(A))_{N=0}^{\infty}$ defined by
$$ x_N(A) = \frac{1}{\log (N+2)} \sum_{j=0}^N \mu(j,A).$$
It is not hard to check that for all all positive integer $k $, 
\begin{equation} \label{eq:two-limits} \lim_{N \to \infty} x_N(A)-x_{kN}(A) = 0. \end{equation}
Let $E \subset \ell_{\infty}$ be the subspace
$$ E = {\rm span } \left\{ \sigma_k( \{x_N(A)\}) \, : \,  k \ge 1, A \in \mathcal{L}_{1,\infty} \right\}.$$
It follows from \eqref{eq:two-limits} that the equation $\omega \circ \sigma_k (x) = \omega(x)$ is satisfied for $x \in E$.
By a version of the Hahn--Banach theorem (see \cite{Edwards}, Theorem 3.3.1), the linear functional $\omega_{|E}$ can be extended to a generalized limit
$\omega' : \ell_{\infty} \to \mathbf{R}$ which is scale-invariant. The usual argument (\cite{LSZ}, Theorem 1.3.1) implies that 
${\rm Tr}_{\omega'}$ (which coincides with ${\rm Tr}_{\omega}$ on $\mathcal{L}_{1,\infty}$) is additive on $\mathcal{L}_{1,\infty}$.
\end{proof}

We also need a version of Fubini's theorem for Dixmier traces.

\begin{thm}\label{little fubini} 
For every $A\in\mathcal{L}_{1,\infty}$ and for every $C\in\mathcal{L}_1,$ we have $A\otimes C\in\mathcal{L}_{1,\infty}$ and
\begin{equation}\label{first assertion} \| A \otimes C \|_{1,\infty} \leq \|A\|_{1,\infty} \|C\|_1 .
\end{equation}
Moreover, for every Dixmier trace ${\rm Tr}_{\omega}$ on $\mathcal{L}_{1,\infty}$, we have 
\begin{equation}\label{second assertion}{\rm Tr}_{\omega}(A\otimes C)={\rm Tr}_{\omega}(A){\rm Tr}(C).
\end{equation}
\end{thm}

\begin{proof} 
We may assume $\|A\|_{1,\infty} = 1$. Recall that $A_0={\rm diag}(\{1,\frac12,\frac13,\cdots\})$. 
We have for all $k \geq 0$,
$$\mu(k,A\otimes C)\leq\mu(k,A_0\otimes C)\leq\frac{1}{k+1}\sum_{j=0}^k\mu(j,C)\leq\frac{\|C\|_1}{k+1}.$$
where the second inequality follows from Proposition 3.14 in \cite{DFWW}. This proves \eqref{first assertion}.

Observe that both sides of \eqref{second assertion} depend linearly on $A$ and $C$ 
(thanks to Theorem \ref{dixmier-traces}). Thus, we can assume without loss of generality that $A,C\geq0.$ 
When $C$ is a rank one projection, \eqref{second assertion} follows from Theorem \ref{dixmier-traces} since in that case $\mu(k,A \otimes C) = \mu(k,A)$ for all $k \geq 0$. Again appealing to
linearity of Dixmier traces, we infer the result for the finite rank operator $C$ and when 
$A \in \mathcal{L}_{1,\infty}$ is arbitrary.
Consider now a general $C \in \mathcal{L}_1$ and let $(C_n)$ be a sequence of finite rank operators such that $\|C-C_n\|_1 \to 0$. We
have
$$|{\rm Tr}_{\omega}(A\otimes C_n)-{\rm Tr}_{\omega}(A\otimes C)|\leq\|A\otimes (C-C_n)\|_{1,\infty}\leq\|A\|_{1,\infty}\|C-C_n\|_1$$
and this quantity tends to $0$ as $n$ goes to infinity. Consequently,
\[{\rm Tr}_{\omega}(A \otimes C) = \lim_{n \to \infty} {\rm Tr}_{\omega}(A \otimes C_n) = \lim_{n \to \infty} {\rm Tr}_{\omega}(A)
{\rm Tr} (C_n) = {\rm Tr}_{\omega}(A) {\rm Tr} (C). \qedhere \]
\end{proof}

As a corollary, we obtain that Dixmier traces give necessary conditions for catalysis.

\begin{cor} \label{corollary-dixmiertraces}
Let $0\leq A \in \mathcal{L}_{1,\infty}$ and $0\leq B \in \overline{{\rm Catal}}(A,\mathcal{L}_{1,\infty}).$ 
Then for every Dixmier trace ${\rm Tr}_{\omega},$ one has 
\begin{equation}\label{inequality-dixmier}
{\rm Tr}_{\omega}(B)\leq{\rm Tr}_{\omega}(A).
\end{equation}
\end{cor}

\begin{proof}
We know from \eqref{continuity} that Dixmier traces are continuous on $\mathcal{L}_{1,\infty}$, and therefore we may
assume that $B \in {\rm Catal}(A,\mathcal{L}_{1,\infty})$. 
By definition of the latter set (see Question \ref{meaningful question}),
there exists a nonzero positive $C$ in $\mathcal{L}_1$ with the property that $B \otimes C \prec \prec A \otimes C$. 
Combining the definition of Hardy-Littlewood submajorization $\prec\prec$ and the definition of a Dixmier trace ${\rm Tr}_{\omega}$ (see Theorem \ref{dixmier-traces}) we infer that the 
inequality  ${\rm Tr}_{\omega}( B \otimes C) \leq {\rm Tr}_{\omega}(A \otimes C)$ holds for every Dixmier trace
${\rm Tr}_{\omega}$. Inequality \eqref{inequality-dixmier} now follows from \eqref{second assertion} and from the fact that ${\rm Tr} (C) > 0$. 
\end{proof}

\section{The case of $\mathcal{L}_{1,\infty}$: the main argument} \label{section:proof}

Here is the main technical result used in the proof of Theorem \ref{main}.
In the lemma below, we tacitly identify a sequence in the space $\ell_{\infty}$ with the corresponding diagonal operator.
For $I \subset \mathbb{N}$, we note by $\chi_I$ the sequence defined by $\chi_I(n)=1$ if $n \in I$ and $\chi_I(n)=0$ otherwise.

\begin{lem} \label{lem:computations}
Let $I$ be the subset of $\mathbb{N}$ defined as 
$$I=\bigcup_{n \geq 0} [ 2^{2n}, 2^{2n+1} ) .$$
Consider the operator\footnote{In the subsequent formulas, the symbol $\oplus$ stands for the direct sum of operators.}
$$B=\bigoplus_{m \in I} 2^{-m} \chi_{[0,2^m)}.$$
Then $B \in \mathcal{L}_{1,\infty}$. Moreover,
\begin{equation} \label{eq:computations}
\limsup_{s \to 0+} s {\rm Tr} ( B^{1+s} ) \leq \frac{5}{9 \log 2} < \frac{2}{3 \log 2} \leq \limsup_{N \to \infty} \frac{1}{\log N} \sum_{k=0}^N \mu(k,B).
\end{equation}
\end{lem}

Let us postpone the proof of Lemma \ref{lem:computations} and show how it implies  
the result stated in Theorem \ref{main}.
Consider $B$ as in Lemma \ref{lem:computations} and fix a number 
$\alpha$ such that $\frac{5}{9 \log 2} < \alpha < \frac{2}{3 \log 2}$. Recall that $A_0 = {\rm diag}(\{1,\frac12,\frac13,\cdots\})$.
Since 
$$ \lim_{s \to 0+} s {\rm Tr} ( (\alpha A_0)^{1+s} ) = \lim_{s \to 0+} s \zeta(s) \alpha^{1+s} = \alpha,$$
it follows from \eqref{eq:computations} that there exists $\delta > 0$ such that the inequality
\begin{equation} \label{eq:traces-A-B}
{\rm Tr} (B^{1+s}) \leq {\rm Tr} ((\alpha A_0)^{1+s})
\end{equation}
holds whenever $0<s\leq\delta.$ Define the operator $A = \alpha A_0 \oplus \|B\|_{1+\delta} p$
where $p$ is a rank one projection. We claim that $B \in {\rm PM}(A,\mathcal{L}_{1,\infty})$: indeed, for $s > \delta$ we
may write 
$$ {\rm Tr} (B^{1+s}) = {\rm Tr} \left( (B^{1+\delta})^{\frac{1+s}{1+\delta}} \right) \leq 
\left( {\rm Tr} (B^{1+\delta}) \right)^{\frac{1+s}{1+\delta}} = \|B\|_{1+\delta}^{1+s} \leq {\rm Tr} (A^{1+s})$$
while for $0<s \leq \delta$ the inequality ${\rm Tr} (B^{1+s}) \leq {\rm Tr} (A^{1+s})$ follows immediately from 
\eqref{eq:traces-A-B}.

We now assume by contradiction that $B$ belongs to the set $\overline{\rm Catal}(A,\mathcal{L}_{1,\infty})$. 
We know from Corollary \ref{corollary-dixmiertraces} that  ${\rm Tr}_{\omega}(B)\leq {\rm Tr}_{\omega}(A)$ 
for every Dixmier trace ${\rm Tr}_{\omega}$. Observing that any such trace vanishes on finite rank operators, we 
see that the value ${\rm Tr}_{\omega}(A)$ coincides with ${\rm Tr}_{\omega}(\alpha A_0)$ and hence is equal to 
$\alpha$ for for every Dixmier trace ${\rm Tr}_{\omega}$ (see the definition given in Theorem \ref{dixmier-traces}). On the other hand, we may choose a generalized limit $\omega$ such that
$$ {\rm Tr}_{\omega}(B) = \limsup_{N \to \infty} \frac{1}{\log N} \sum_{k=0}^N \mu(k,B)$$
and obtain from \eqref{eq:computations} that $\frac{2}{3 \log 2} \leq \alpha$, a contradiction.

We note that the Dixmier trace considered in the proof does not behave in a monotone way with respect to trace of powers:
we have ${\rm Tr}(B^p) \leq {\rm Tr}(A^p)$ for every $p > 1$, but ${\rm Tr}_{\omega}(B) > {\rm Tr}_{\omega}(A)$.

\section{Proof of Lemma \ref{lem:computations}} \label{section:computations}

Let $I$ and $B$ be as defined in Lemma \ref{lem:computations}, and denote by $E_B$ the spectral measure of $B$. 
First, note that for every integer $m$,
\begin{equation} \label{trace-B} {\rm Tr}(E_B(2^{-m},\infty))\leq\sum_{l<m}2^l\leq 2^m. \end{equation}
Hence, for every positive integer $n$, writing $2^m\leq n<2^{m+1}$, we infer
\begin{equation}\label{equivalently}{\rm Tr}(E_B(\frac 1n,\infty))\leq {\rm Tr}(E_B((2^{-m-1},\infty))\leq 2^{m+1}\leq 2n.
\end{equation}
Recall also (see e.g. \cite[Chapter 2, Section 2.3]{LSZ}) that $\mu(k,B)$, $ k\ge 0$ can be computed via the formula
$$
\mu(k,B)=\inf\{s\ge 0:\ {\rm Tr}(E_{|B|}((s,\infty))\leq k\}
$$
Hence, it follows from \eqref{equivalently} that $\mu(k,B)\leq\frac{2}{k+1}$ for every $k\geq0$ and, in particular, 
$B\in\mathcal{L}_{1,\infty}.$
We now prove the right inequality in \eqref{eq:computations}. For a given $n$, let $N={\rm Tr}(E_B(2^{-2^{2n+1}},\infty))$. We
know from \eqref{trace-B} that $N \leq 2^{2^{2n+1}}$. Therefore,
$$\sum_{k=0}^{N-1}\mu(k,B)={\rm Tr}(BE_B(2^{-2^{2n+1}},\infty))= {\rm card} ( I \cap [0,2^{2n+1}] )
=\frac23\cdot 2^{2n+1} -\frac13.$$
Hence, for $N$ as above, we have
$$\frac1{\log(N)}\sum_{k=0}^{N-1}\mu(k,B)\geq\frac1{\log(2^{2^{2n+1}})}\cdot\Big(\frac23\cdot 2^{2n+1}-\frac13\Big)=\frac{2}{3\log(2)}+o(1).$$
as needed.

We now focus on the left inequality in \eqref{eq:computations} and use the following summation formula, whose proof
we postpone.
For a given sequence $(x_n) \in \ell_{\infty}$ and for a given $s>0,$ we have that
\begin{equation} \label{summation} \sum_{m=0}^{\infty} \left( \sum_{k=0}^m \sum_{l=0}^k x_l \right) 2^{-ms} 
= (1-2^{-s})^{-2} \sum_{l=0}^{\infty} x_l 2^{-ls} .
\end{equation}

Note that ${\rm Tr} (B^{1+s}) = \sum_{m \in I} 2^{-ms} = \sum_{m \geq 0} \chi_I(m) 2^{-ms}$ (here, $\chi_I(0) =0$). Applying
\eqref{summation} to $x = \chi_I$, we obtain, for every $M > 0$
\begin{align*}
\limsup_{s\to0+}s&\sum_{l\geq0}\chi_I(l)2^{-ls}=\limsup_{s\to0+}s(1-2^{-s})^2\sum_{m\geq0}(\sum_{k=0}^m\sum_{l=0}^k \chi_I(l))2^{-ms}\cr&
=\limsup_{s\to0+}s(1-2^{-s})^2\sum_{m\geq M}\Big(\frac1{(m+1)^2}\sum_{k=0}^m\sum_{l=0}^k\chi_I(l)\Big)\cdot (m+1)^22^{-ms}\cr&
\leq\Big(\sup_{m\geq M}\frac1{(m+1)^2}\sum_{k=0}^m\sum_{l=0}^k \chi_I(l)\Big)\cdot\Big(\limsup_{s\to0+}s(1-2^{-s})^2\sum_{m\geq M}(m+1)^22^{-ms}\Big).
\end{align*}
Passing $M\to\infty,$ we infer that
$$\limsup_{s\to0+}s\sum_{l\geq0}\chi_I(l)2^{-ls}\leq C  \limsup_{s\to0+}s(1-2^{-s})^2\sum_{m\geq 0}(m+1)^22^{-ms},$$
where 
$$C := \limsup_{m\to\infty}\frac1{(m+1)^2} \sum_{k=0}^m\sum_{l=0}^k\chi_I(l).$$
An elementary computation gives
$$ \sum_{m=0}^{\infty} (m+1)^2 2^{-ms} = \frac{1+2^{-s}}{(1-2^{-s})^3}.$$
It follows that
$$ \limsup_{s \to 0+} s {\rm Tr} (B^{1+s}) \leq \frac{2C}{\log 2}.$$
It remains to show that $C \leq 5/18$ (we actually show $C=5/18$). To that end, we think of $\chi_I$ as an element of $L_{\infty}(0,\infty)$
and define $z\in L_{\infty}(0,\infty)$ by setting $z=\chi_{\bigcup_{n\in\mathbb{Z}}[2^{2n},2^{2n+1})}.$ 
Observe that $\chi_I \leq z.$ Therefore, 
$$C \leq \limsup_{t\to\infty}\frac1{t^2}\int_0^t\int_0^sz(u) \, \mathrm{d}u \mathrm{d}s.$$
Since $z(4t)=z(t)$ for every $t>0,$ applying Fubini's theorem we have 
$$C \leq \sup_{t\in(1,4)}\frac1{t^2}\int_0^tz(u)(t-u) \,\mathrm{d}u.$$
However,
$$\frac1{t^2}\int_0^tz(u)(t-u) \,\mathrm{d}u=
\begin{cases}
\frac{1}{2}-\frac{2}{3t}+\frac{2}{5t^2},\quad 1\leq t\leq 2\\
\frac{4}{3t}-\frac{8}{5t^2},\quad 2\leq t\leq 4
\end{cases}
$$
Hence, the latter supremum is, in fact, a maximum which is attained at $t=\frac{12}{5}$ and equal to $\frac{5}{18}$.

\begin{proof}[Proof of \eqref{summation}]
Write 
\begin{align*}\sum_{m\geq0}(\sum_{k=0}^m\sum_{l=0}^kx_l)2^{-ms}&=\sum_{m\geq k\geq 0}(\sum_{l=0}^kx_l)2^{-ms}=\sum_{k=0}^{\infty}(\sum_{l=0}^kx_l)\sum_{m=k}^{\infty}2^{-ms}\cr &
=(1-2^{-s})^{-1}\sum_{k=0}^{\infty}(\sum_{l=0}^kx_l)2^{-ks}=(1-2^{-s})^{-1}\sum_{k\geq l\geq 0}x_l2^{-ks}\cr
&=
(1-2^{-s})^{-1}\sum_{l=0}^{\infty}x_l\sum_{k=l}^{\infty}2^{-ks}=(1-2^{-s})^{-2}\sum_{l=0}^{\infty}x_l2^{-ls}. \qedhere
\end{align*}
\end{proof}


\begin{thebibliography}{99}
\bibitem{Au-Nec} Aubrun G., Nechita I. {\it Catalytic majorization and $l_p$ norms.} Comm. Math. Phys. {\bf 278} (2008), no. 1, 133--144.
\bibitem{CDSS} Chilin V., Dodds P., Sedaev A., Sukochev F. {\it Characterisations of Kadec-Klee properties in symmetric spaces of measurable functions.} Trans. Amer. Math. Soc. {\bf 348} (1996), no. 12, 4895--4918.
\bibitem{Calkin} Calkin J. {\it Two-sided ideals and congruences in the ring of bounded operators in Hilbert space,} Ann. of Math. (2) {\bf 42}, (1941) 839--873.
\bibitem{CDS} Chilin V., Dodds P., Sukochev F. {\it The Kadec-Klee property in symmetric spaces of measurable operators.} Israel J. Math. {\bf 97} (1997), 203--219.
\bibitem{DFWW} Dykema K., Figiel T., Weiss G., Wodzicki M. {\it Commutator structure of operator ideals.} Adv. Math. {\bf 185} (2004), no. 1, 1--79.
\bibitem{Edwards} Edwards R. {\it Functional Analysis. Theory and applications.}  Corrected reprint of the 1965 original. Dover Publications, Inc., New York, 1995
\bibitem{JP} Jonathan D., Plenio M. {\it Entanglement-assisted local manipulation of pure quantum states} Phys. Rev. Lett. {\bf 83}  (1999),  no. 17, 3566--3569.
\bibitem{KS} Kalton N,  Sukochev F.  {\it Symmetric norms and spaces of operators,} J. Reine Angew. Math. {\bf 621} (2008), 81--121.
\bibitem{Klimesh} Klimesh M. {\it Inequalities that completely characterize the Catalytic Majorization Relation}, arXiv eprint \texttt{0709.3680} 
\bibitem{LT1} Lindenstrauss J.,  Tzafriri L. \textit{Classical Banach Spaces I}, Springer-Verlag, 1977.
\bibitem{LSZ} Lord S., Sukochev F., Zanin D. {\it Singular traces. Theory and applications.} De Gruyter Studies in Mathematics, {\bf 46}. De Gruyter, Berlin, 2013. xvi+452 pp.
\bibitem{SS} Sedaev A., Sukochev F. {\it Dixmier measurability in Marcinkiewicz spaces and applications.} J. Funct. Anal. {\bf 265} (2013), no. 12, 3053--3066.
\bibitem{Simon} Simon B. {\it Trace ideals and their applications,} Second edition. Mathematical Surveys and Monographs, {\bf 120}. American Mathematical Society, Providence, RI, 2005.
\bibitem{S} Sukochev F.  {\it Completeness of quasi-normed symmetric operator spaces}.  Indag. Math. (N.S.) {\bf 25} (2014), no. 2, 376--388.
\bibitem{Turgut} Turgut S. {\it Catalytic transformations for bipartite pure states.} J. Phys. A: Math. Theor. {\bf 40} (2007), 12185--12212.
\end{thebibliography}
\end{document}